\documentclass[12pt]{amsart}
\usepackage{verbatim}
\usepackage{amssymb, amsmath}
\usepackage{graphicx}
\usepackage{caption}
\usepackage{subcaption}
 \usepackage{enumitem}
\usepackage{color}
\usepackage{url}

 \usepackage[colorlinks,linkcolor=blue,anchorcolor=blue,citecolor=blue,backref=page]{hyperref}

\usepackage{hyperref}

\begin{document}
\newtheorem{thm}{Theorem}[section]
\newtheorem*{thm*}{Theorem}
\newtheorem{lem}[thm]{Lemma}
\newtheorem{prop}[thm]{Proposition}
\newtheorem{cor}[thm]{Corollary}
\newtheorem{conj}{Conjecture}
\newtheorem{proj}[thm]{Project}
\newtheorem{question}{Question}
\newtheorem{rem}{Remark}[section]

\theoremstyle{definition}
\newtheorem*{defn}{Definition}
\newtheorem*{remark}{Remark}
\newtheorem{exercise}{Exercise}
\newtheorem*{exercise*}{Exercise}

\numberwithin{equation}{section}

\newcommand{\rad}{\operatorname{rad}}

\newcommand{\Z}{{\mathbb Z}} 
\newcommand{\Q}{{\mathbb Q}}
\newcommand{\R}{{\mathbb R}}
\newcommand{\C}{{\mathbb C}}
\newcommand{\N}{{\mathbb N}}
\newcommand{\FF}{{\mathbb F}}
\newcommand{\fq}{\mathbb{F}_q}
\newcommand{\rmk}[1]{\footnote{{\bf Comment:} #1}}

\renewcommand{\mod}{\;\operatorname{mod}}
\newcommand{\ord}{\operatorname{ord}}
\newcommand{\TT}{\mathbb{T}}
\renewcommand{\i}{{\mathrm{i}}}
\renewcommand{\d}{{\mathrm{d}}}
\renewcommand{\^}{\widehat}
\newcommand{\HH}{\mathbb H}
\newcommand{\Vol}{\operatorname{vol}}
\newcommand{\area}{\operatorname{area}}
\newcommand{\tr}{\operatorname{tr}}
\newcommand{\norm}{\mathcal N} 
\newcommand{\intinf}{\int_{-\infty}^\infty}
\newcommand{\ave}[1]{\left\langle#1\right\rangle} 
\newcommand{\Var}{\operatorname{Var}}
\newcommand{\Prob}{\operatorname{Prob}}
\newcommand{\sym}{\operatorname{Sym}}
\newcommand{\disc}{\operatorname{disc}}
\newcommand{\CA}{{\mathcal C}_A}
\newcommand{\cond}{\operatorname{cond}} 
\newcommand{\lcm}{\operatorname{lcm}}
\newcommand{\Kl}{\operatorname{Kl}} 
\newcommand{\leg}[2]{\left( \frac{#1}{#2} \right)}  
\newcommand{\Li}{\operatorname{Li}}

\newcommand{\sumstar}{\sideset \and^{*} \to \sum}

\newcommand{\LL}{\mathcal L} 
\newcommand{\sumf}{\sum^\flat}
\newcommand{\Hgev}{\mathcal H_{2g+2,q}}
\newcommand{\USp}{\operatorname{USp}}
\newcommand{\conv}{*}
\newcommand{\dist} {\operatorname{dist}}
\newcommand{\CF}{c_0} 
\newcommand{\kerp}{\mathcal K}

\newcommand{\Cov}{\operatorname{cov}}
\newcommand{\Sym}{\operatorname{Sym}}

\newcommand{\Ht}{\operatorname{Ht}}

\newcommand{\E}{\operatorname{\mathbb E}} 
\newcommand{\sign}{\operatorname{sign}} 
\newcommand{\meas}{\operatorname{meas}} 
\newcommand{\length}{\operatorname{length}} 

\newcommand{\divid}{d} 

\newcommand{\GL}{\operatorname{GL}}
\newcommand{\SL}{\operatorname{SL}}
\newcommand{\Sp}{\operatorname{Sp}}
\newcommand{\re}{\operatorname{Re}}
\newcommand{\im}{\operatorname{Im}}
\newcommand{\res}{\operatorname{Res}}
 \newcommand{\eigen}{\Lambda} 
\newcommand{\tens}{\mathbf t} 
\newcommand{\diam}{\operatorname{diam}}
\newcommand{\fixme}[1]{\footnote{Fixme: #1}}
 \newcommand{\EWp}{\mathbb E^{\rm WP}_g} 
\newcommand{\orb}{\operatorname{Orb}}
\newcommand{\supp}{\operatorname{Supp}}
\newcommand{\mmfactor }{\textcolor{red}{c_{\rm Mir}}}
\newcommand{\Mg}{\mathcal M_g} 
\newcommand{\MCG}{\operatorname{Mod}} 
\newcommand{\Diff}{\operatorname{Diff}} 
\newcommand{\If}{I_f(L,\tau)}

\newcommand{\GOE}{\operatorname{GOE}}
\newcommand{\GUE}{\operatorname{GUE}}
\newcommand{\GSE}{\operatorname{GSE}}
\newcommand{\rest}{\operatorname{rest}} 
\newcommand{\diag}{\operatorname{DIAG}} %
\newcommand{\antidiag}{\operatorname{ANTI-DIAG}}
\newcommand{\off}{\operatorname{OFF}} %
\newcommand{\finitefieldq}{r} 

\title[Closed geodesics in homology classes]{Closed geodesics in homology classes on random hyperbolic surfaces of large genus}
\author{Ze\'ev Rudnick}
\address{School of Mathematical Sciences, Tel Aviv University, Tel Aviv 69978, Israel} 
\address{Mathematical Research Center, Shandong University,
Jinan 250100, P.R. China.}
\email{rudnick@tauex.tau.ac.il}

\date{\today}
 \begin{abstract}
 We study the distribution of closed geodesics in homology classes on random hyperbolic surfaces of large genus. Viewing the surface as a random point in moduli space equipped with the Weil--Petersson probability measure, we investigate the fluctuations of the weighted counting function of closed geodesics in homology classes modulo $q$. We show that, in the large genus limit, the variance is asymptotic to $X\log X$ for every modulus $q>2$, with an exceptional factor of two when $q=2$. We relate our findings to a conjecture of Hooley, and to work of Friedlander and Goldston, 
 on primes in arithmetic progressions. We also introduce a statistical model based on the random allocation of weighted balls which exhibits analogous behavior.
 \end{abstract}
\maketitle

\tableofcontents

 \section{Introduction}

Prime numbers and primitive closed geodesics are among the most fundamental
discrete objects in mathematics. Although one belongs to arithmetic and the
other to geometry, they have long been known to obey remarkably parallel
asymptotic laws. The Prime Number Theorem has its geometric counterpart in
the Prime Geodesic Theorem. One particularly natural aspect of this analogy concerns the
distribution of primes in arithmetic progressions and the distribution of
closed geodesics among the homology classes of a surface.

Let $M$ be a closed hyperbolic surface of genus $g\ge2$. Every oriented
closed geodesic determines a homology class in $H_1(M,\mathbb Z)$. Reducing
modulo an integer $q>1$ gives a class in the finite group
$H_1(M,\mathbb Z/q\mathbb Z)$, making it natural to ask how primitive closed
geodesics are distributed among these homology classes. Phillips and Sarnak
\cite{PS} established the analogue of the Prime Number Theorem
for arithmetic progressions, proving that for every fixed modulus $q$,
primitive closed geodesics become equidistributed among the homology classes
modulo $q$.

Once equidistribution is known, the next natural question concerns the size
of the fluctuations around this average. This is the geometric analogue of
one of the central problems of analytic number theory. 
The purpose of this paper is to determine the variance of the distribution of closed geodesics among homology classes on random hyperbolic surfaces and to compare it with the corresponding arithmetic problem. 

Unlike the arithmetic setting, hyperbolic geometry admits averaging over natural ensembles of surfaces. In this paper we average over the Weil–Petersson probability measure on the moduli space of hyperbolic surfaces and study the large-genus limit. This additional averaging makes the corresponding problem accessible.

We consider all geodesics to be oriented, unless stated otherwise. Define
the norm of a closed geodesic by
\[
N(\gamma)=e^{\ell(\gamma)},
\]
where $\ell(\gamma)$ denotes its length. A geodesic is primitive if it is
not an iterate of another geodesic. If $\gamma=\gamma_0^k$, where $\gamma_0$
is primitive, and $k\geq 1$, define the geometric von Mangoldt function by
\[
\Lambda(\gamma)=\ell(\gamma_0).
\]
The corresponding Chebyshev function is
\[
\Psi_M(X)=\sum_{N(\gamma)\le X}\Lambda(\gamma),
\]
where the sum is over oriented closed geodesics. The Prime Geodesic Theorem
asserts that
\[
\Psi_M(X)\sim X,\qquad X\to\infty.
\]

For an integer $q>1$ and a homology class
$\alpha\in H_1(M,\mathbb Z)$, define
\[
\Psi_M(X;q,\alpha)
=
\sum_{\substack{N(\gamma)\le X\\
[\gamma]\equiv\alpha \bmod q}}
\Lambda(\gamma),
\]
the weighted counting function of oriented closed geodesics $\gamma$ whose homology
class $[\gamma]\in H_1(M,\Z)$ is congruent to $\alpha$ modulo $q$, after we identify the integral homology group $H_1(M,\Z)$ with the lattice $\Z^{2g}$, see \S~\ref{subsec:homology}. 

By definition, the average of this quantity over all $\alpha \bmod q$ is $$\frac{\Psi_M(X)}{q^{2g}}$$ 
which is asymptotic to $X/q^{2g}$ as $X\to \infty$. 
Phillips and Sarnak \cite{PS} proved that for every fixed modulus $q$,
\[
\Psi_M(X;q,\alpha)\sim \frac{X}{q^{2g}},\quad X\to \infty.
\]

To measure the fluctuations around the mean, define the (unnormalized)
variance
\[
G_M(X,q)
=
\sum_{\alpha\;({\rm mod}\;q)}
\left|
\Psi_M(X;q,\alpha)
-
\frac{\Psi_M(X)}{q^{2g}}
\right|^2.
\]
This is an analogue of a similar quantity for primes in arithmetic progressions, studied by Hooley \cite{HooleyICM}, Friedlander and Goldston \cite{FG}, Keating and Rudnick \cite{KeatingRudnick},  and others \cite{Fiorilli, HKR}, see section~\ref{sec:Hooley et al}. 

Our main result determines the expected value of $G_M(X,q)$ over moduli
space in the large genus limit.

\begin{thm}\label{thm:G}
As $X\to\infty$,
\[
\lim_{g\to\infty}\EWp(G_M(X,q))
\sim
\begin{cases}
X\log X,&q>2,\\
2X\log X,&q=2.
\end{cases}
\]
\end{thm}

The factor $X$ is consistent with square-root fluctuations in the
distribution of closed geodesics among the homology classes, as in the arithmetic setting. At first sight the logarithmic factor $X\log X$ contrasts  
with the behavior predicted for primes in arithmetic progressions, see Section~\ref{sec:Hooley et al}. 
However, the order of limits in Theorem~\ref{thm:G}  places the geometric problem in a sparse regime: for fixed $X$ and $q$, the number $q^{2g}$ 
 of homology classes tends to infinity with $g$, while the number of geodesics being distributed remains of order depending only on $X$. As we explain in Section~\ref{sec:Hooley et al}, the appropriate arithmetic comparison is therefore with the “trivial” range $X<Q$, where the corresponding variance is also $X\log X$.
 

In Section~\ref{sec:random} we suggest a statistical model, of random allocation of weighted balls, to compare with our results.


\section{Background}

\subsection{Homology} \label{subsec:homology}
Let $M$ be a closed hyperbolic surface of genus $g\geq 2$. Each nontrivial  free homotopy class contains a unique closed geodesic, allowing us, once we have chosen a base point $x_0\in M$, to identify the set of closed geodesics with the nontrivial conjugacy classes in the fundamental group $\pi_1(M,x_0)$.

We choose a canonical homology basis, which is a system of simple closed curves $\alpha_1,\ldots, \alpha_g$, $\beta_1,\ldots, \beta_g$, each $\alpha_i$ intersects $\beta_i$ exactly at one point, and there are no other intersections (Figure~\ref{fig:canonical homology basis}).   Choosing a common base point $x_0\in M$ and connecting it to each of the curves $\alpha_1,\ldots, \beta_g$ by a path $\eta_j$ gives a set of loops $A_1,\ldots, A_g, B_1,\ldots, B_g$ which generate the fundamental group $\pi_1(M,x_0)$, with the sole relation being that the product of the commutators $[A_i,B_i]=A_iB_iA_i^{-1}B_i^{-1}$ is the identity:
\[
[A_1,B_1]\cdot \ldots \cdot [A_g,B_g]=1.
\]
We then get a map $e:\pi_1(M,x_0)\to \Z^{2g}$ mapping each word in the generators to the sum of the exponents with which each generator appears. For instance, for $g=2$ if $w=A_1B_1^2A_2B_1^3$ then $e(w)=(1,5,1,0)\in \Z^4$. This gives an isomorphism $H_1(M,\Z)\simeq \pi_1(M,x_0)^{ab}\to \Z^{2g}$. 
Thus the coordinates of  $e(\gamma)$  may be viewed as the generalized winding numbers of $\gamma$ around the handles of the surface.

Equivalently, let $\omega_1,\ldots, \omega_{2g}$ be the basis of harmonic one-forms dual to our canonical homology basis: $\int_{\alpha_j}\omega_j=\int_{\beta_j} \omega_{j+g}=1$, $j=1,\ldots, g$, and all other period integrals vanish. Then  our  map $e$ coincides with the period map
\[
[\gamma]\in H_1(M,\Z)\mapsto \left(\int_{\gamma}\omega_1, \ldots, \int_{\gamma}\omega_{2g} \right).
\]
 
 Given an integer $q\geq 1$, the mod $q$ homology  can be identified with homology with coefficients in $\Z/q\Z$
\[
H_1(M,\Z/q\Z) \simeq H_1(M,\Z)/q H_1(M,\Z) \simeq \left( \Z/q\Z \right)^{2g}.
\]


The intersection form is a nondegenerate symplectic form on $H_1(M,\Z)$. 
The matrix of the intersection form in the canonical homology basis is 
\[
 \begin{pmatrix} 0 & I_g\\ -I_g&0\end{pmatrix}.
\]
The mapping class group preserves the intersection form, and thus is mapped to the integral symplectic group $\Sp(2g,\Z)$. This map is in fact surjective, see e.g. \cite[Theorem 6.4]{FM}. 

A particularly important class of closed geodesics for our purposes is that of simple (embedded, or equivalently, having no self-intersections) non-separating geodesics, namely those whose complement in the surface is connected. Such a geodesic represents a primitive homology class. Moreover, the mapping class group acts transitively on primitive homology classes (see e.g.  \cite[\S 1.3.1]{FM}).  

\begin{figure}[ht]
\begin{center}
\includegraphics[height=50mm, width=120mm]{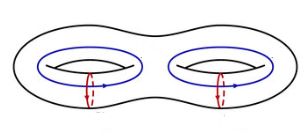}
\caption{A canonical homology basis for a genus $2$ surface.}
\label{fig:canonical homology basis}
\end{center}
\end{figure}

\subsection{Mirzakhani's integration formulas}


Let $S_g$ be a topological compact surface of genus $g$, and $\Mg$ the moduli space of hyperbolic metrics on $S_g$, equipped with the Weil-Petersson probability measure. 

Given a reasonable function $F$ on $(0,\infty)$, we want to compute the expected value over the moduli space $\Mg$ of 
\[
F_{SNS}(M):=\sum_{\gamma\; SNS} F(\ell_M(\gamma))
\]  
where $M\in \Mg$, the sum is over all simple closed (unoriented) geodesics $\gamma$, which are non-separating, that is, the complement  $S_g\backslash \gamma\simeq S_{g-1,2}$ is a connected surface, of genus $g-1$ with two boundary components, and $\ell_M(\gamma)$ is the length of the geodesic. Combining Mirzakhani's integration formula (see \cite[Theorem 2.2]{MP}) with asymptotics of volume ratios \cite[Proposition 3.1]{MP} and  \cite[Lemma 22]{NieWuXue} gives
\begin{equation}\label{eq:sum SNS}
\lim_{g\to \infty} \EWp(F_{SNS}) = \frac 12 \int_0^\infty F(\ell) \left( \frac{\sinh (\ell/2)}{\ell/2}\right)^2   \,\ell \,d\ell .
\end{equation}

The second situation is when we form the double sum
\[
F_{SNS,2}(M) = \sum_{(\gamma,\gamma') \; SNS} F_1\left( \ell_M(\gamma)\right)  F_2\left( \ell_M(\gamma') \right)
\]
where the sum is over pairs of simple non-separating geodesics, which are non-homotopic and disjoint: $\gamma\cap \gamma'=\emptyset$, and such that the complement $S_g\backslash \gamma\cup \gamma'\simeq S_{g-2,4}$ is still connected, so of genus $g-2$ with $4$ boundary components. Then 
\begin{multline}\label{eq:double sum SNS}
\lim_{g\to \infty} \EWp(F_{SNS,2}) =
\\
 \frac 1{2} \int_0^\infty   F_1(\ell)  \left( \frac{\sinh (\ell/2)}{\ell/2}\right)^2  
   \,\ell \,d\ell  \cdot  \frac 1{2} \int_0^\infty   F_2(\ell)  \left( \frac{\sinh (\ell/2)}{\ell/2}\right)^2  
   \,\ell \,d\ell .
\end{multline}

 \section{The expected value of $G_M(X,q)$} 
\subsection{Preliminaries}
Recall
\[
G_M(X,q) =   \sum_{\alpha \bmod q} \left| \Psi_M(X;q,\alpha) - \frac{\Psi_M(X)}{q^{2g}}  \right|^2 .
\]

\begin{lem}\label{lem:simple reduction}
\[
G_M(X,q) =   \sum_{\substack{N(\gamma_1),N(\gamma_2)\leq X\\ [\gamma_1]=[\gamma_2] \bmod q}} \Lambda(\gamma_1)\Lambda(\gamma_2) -    \frac{\Psi_M(X)^2}{q^{2g}}  
\]
\end{lem}
\begin{proof}
This is a straightforward consequence of the definitions. We have
\[
\frac 1{q^{2g}} G_M(X,q) =  \frac 1{q^{2g}} \sum_{\alpha \bmod q} \Psi_M(X;q,\alpha)^2
-   \left( \frac{\Psi_M(X)}{q^{2g}}\right)^2 
\]
and
\[
\begin{split}
 \sum_{\alpha \bmod q} \Psi_M(X;q,\alpha)^2 &=   \sum_{\alpha \bmod q} \sum_{\substack{ N(\gamma_1)\leq X \\  [\gamma_1]=\alpha \bmod q}}  \Lambda(\gamma_1)\sum_{\substack{N(\gamma_2)\leq X\\ [\gamma_2]=\alpha \bmod q}} \Lambda(\gamma_2)
\\
&=\sum_{\substack{N(\gamma_1),N(\gamma_2)\leq X\\ [\gamma_1]=[\gamma_2] \bmod q}} \Lambda(\gamma_1)\Lambda(\gamma_2) 
\end{split}
\]
which proves our claim. 
\end{proof}

We start with Lemma~\ref{lem:simple reduction} and want to evaluate the limiting expected value of 
\[
J_M(X,q)=\sum_{\substack{N(\gamma_1),N(\gamma_2)\leq X\\ [\gamma_1]=[\gamma_2] \bmod q}} \Lambda(\gamma_1)\Lambda(\gamma_2) 
\]
where the sum is over all oriented geodesics, not necessarily primitive.

We expand as $\gamma_i^{k_i}$ with $\gamma_i$ primitive, $k_i\geq 1$, and $k_1[\gamma_1]=k_2[\gamma_2] \bmod q$. 
Following the computations in \cite{MP, RudGOE, MarklofMonk}, up to an error of $O_{X,q}(1/g)$, it suffices to consider  the following four kinds of pairs of primitive geodesics:   
\begin{enumerate}
\item Diagonal: $\gamma_1=\gamma_2$ with $\gamma_1$  simple and non-separating; 
\item Anti-diagonal: $\gamma_2=\gamma_1^{-1}$   with $\gamma_1$  simple and non-separating; 
\item SNS pair: $\gamma_1\cap \gamma_2=\emptyset$  and the complement of $\gamma_1\cup \gamma_2$ is connected;
\item all other pairs.
\end{enumerate}

 \subsection{Diagonal pairs}\label{subsec:diagonal pairs}
 We expand over pairs 
 \[
 \gamma_1=\gamma^{k_1}, \quad \gamma_2=\gamma^{k_2}
 \] 
 with $\gamma$ simple, non-separating and $k_1,k_2\geq 1$. The homology condition is 
 \[
 k_1[\gamma]=k_2[\gamma] \bmod q.
 \]
  For a simple non-separating geodesic, the homology class is a primitive vector in $H_1(M,\Z)\simeq \Z^{2g}$ which may be assumed to be $[\gamma]=(1,0,0,\ldots,0)$ (see e.g.  \cite[\S 1.3.1]{FM}) and so we require 
 $$k_1=k_2 \mod q.$$
 
 We also want $N(\gamma^{k_i}) \leq X$, that is,  the lengths   satisfy $k_i\ell_\gamma\leq \log X$. So we require 
 $$
 \ell_\gamma\leq \frac{\log X}{\max(k_1,k_2)}\;.
 $$  
 
 We omit the requirement that $\gamma$ is oriented, at a cost of multiplying by a factor of $2$. Then we can use \eqref{eq:sum SNS} 
 \[
 \begin{split}
\diag &= 2\sum_{\substack{k_1,k_2\geq 1\\k_1=k_2 \bmod q} } \lim_{g\to \infty} \EWp\left\{ \sum_{\substack{\gamma\; SNS\\ {\rm unoriented}   
 \\  \ell_\gamma\leq \log X/\max(k_1,k_2)}} \ell_\gamma^2 \right\} 
\\
& =  2\sum_{\substack{k_1,k_2\geq 1\\k_1=k_2 \bmod q} }
\frac 12 \int_{0}^{\log X/\max(k_1,k_2)} \ell^2 \left( \frac{\sinh(\ell/2)}{\ell/2} \right)^2 \ell d\ell 
\\
&=  \sum_{\substack{k_1,k_2\geq 1\\k_1=k_2 \bmod q} } I(\frac{\log X}{\max(k_1,k_2)})    \;, 
 \end{split}
 \]
 where
 \[
 I(A):=\int_0^A \left(2\sinh \frac \ell 2 \right)^2 \,\ell d\ell . 
 \]
We have
\begin{equation}\label{eq: exact IA}
I(A)= 2A\sinh A-2\cosh A-A^2+ 2. 
\end{equation}
For $0<A<1$, we use $2\sinh \frac \ell 2\leq 2\ell$ so that
\begin{equation}\label{eq:A less 1}
I(A)\leq 4\int_0^A \ell^3 d\ell = A^4.
\end{equation}
 For $A>1$, we will use $(2\sinh\frac \ell2)^2 \leq e^\ell$ so that 
 \begin{equation}\label{eq:A gtr 1}
 I(A)\leq \int_0^A e^\ell \ell d\ell <Ae^A. 
 \end{equation}

For $(k_1,k_2)=(1,1)$, we evaluate the  contribution by using \eqref{eq: exact IA} as
\[
I(\log X)= X\log X-X+O\left( (\log X)^2 \right) \;.
\]

For all other pairs, we bound the contribution as follows (we don't use the congruence modulo $q$): 

For $ \max(k_1,k_2)>\log X$, we use \eqref{eq:A less 1}
\[
 I(\frac{\log X}{\max(k_1,k_2)})   \ll \frac{(\log X)^4}{(\max(k_1,k_2))^4}
\]
so that 
\[
\begin{split}
\sum_{\max(k_1,k_2)>\log X} &  I(\frac{\log X}{\max(k_1,k_2)})  \ll \sum_{\max(k_1,k_2)>\log X}\frac{(\log X)^4}{(\max(k_1,k_2))^4}
\\
&\ll 
(\log X)^4 \left( \sum_{k>\log X} \frac 1{k^4} + 2\sum_{k_2>\log X} \frac 1{k_2^4} \sum_{1\leq k_1< k_2} 1 \right)
\\
&\ll (\log X)^2. 
\end{split}
 \]

For $ 1<\max(k_1,k_2)\leq \log X$ we use \eqref{eq:A gtr 1} to bound 
\[
\begin{split}
&\sum_{1<\max(k_1,k_2)\leq \log X}     I(\frac{\log X}{\max(k_1,k_2)}) 
 \leq \sum_{1<\max(k_1,k_2)\leq \log X} \frac{X^{1/\max(k_1,k_2)}\log X }{\max(k_1,k_2)}
\\
&\ll \sum_{2\leq k\leq \log X} \frac{X^{1/k}\log X}{k} + 2\sum_{2\leq k_2\leq \log X} \frac{X^{1/k_2}\log X}{k_2} \sum_{1\leq k_1<k_2} 1
\\
&\ll X^{1/2} \log X.
\end{split}
 \]

Altogether, we found that
\[
\diag\sim X\log X.
\]

\subsection{The anti-diagonal}
Now the condition is  $\gamma_1=\gamma^{k_1}$, $\gamma_2=\gamma^{-k_2}$ with $\gamma$ simple, non-separating, $k_1,k_2\geq 1$ and 
\[
k_1[\gamma] = -k_2[\gamma] \bmod q.
\]
As before, we may assume $[\gamma]=(1,0,\ldots, 0)$ so that we need
\[
k_1=-k_2 \mod q.
\]

If $q=2$ then we recover the same condition as the diagonal, and in that case
\[
\antidiag \sim X\log X, \quad q=2. 
\]
For $q>2$, we no longer have the pair $(k_1,k_2)=(1,1)$ and then the other pairs are bounded as in the case of the diagonal, to give
\[
\antidiag \ll X^{1/2} \log X, \quad q>2. 
\]

\subsection{SNS pairs} 
Now we sum over  pairs $(\gamma_1^{k_1},\gamma_2^{k_2})$ where $\gamma_1,\gamma_2$ are simple non-intersecting so that the complement of $\gamma_1\cup \gamma_2$ is connected, and $k_1,k_2\geq 1$, and $k_1[\gamma_1]=k_2[\gamma_2] \bmod q$. 
By the change of coordinates principle (see \cite[\S~1.3.3 Example 3]{FM}), we may assume that
\[
[\gamma_1]=(1,0,0,\ldots,0), \quad [\gamma_2]=(0,1,0,\ldots ,0)
\]
so that the condition $k_1[\gamma_1]=k_2[\gamma_2] \bmod q$ is equivalent to 
\[
k_1,k_2=0\bmod q.
\]
In  particular $k_1,k_2\geq q \geq 2$. 

The contribution of the  SNS pairs is therefore bounded by 
\[
\sum_{\substack{k_1,k_2\geq q\\k_1=k_2 =0\bmod q}}  J(k_1,k_2)
\]
where, by \eqref{eq:double sum SNS}, 
\[
\begin{split}
  J(k_1,k_2) & = \lim_{g\to \infty} \EWp\left\{ \sum_{\substack{(\gamma_1,\gamma_2) \, SNS\\  k_1\ell_1,k_2\ell_2\leq \log X}} 
  \ell_1 \cdot \ell_2 \right\}
  \\
  &  = 2^2 \frac 1{2  } \int_0^{\log X/k_1} \ell_1 \left( \frac{\sinh(\ell_1/2)}{\ell_1/2} \right)^2 \,\ell_1 \,d\ell_1
  \\
  & \cdot 
\frac 1{2  }    \int_0^{\log X/k_2}\ell_2  \left( \frac{\sinh(\ell_2/2)}{\ell_2/2} \right)^2 \,\ell_2 \,d\ell_2
   \\
  & = I_1(\frac{\log X}{k_1}) \cdot I_1(\frac{\log X}{k_2})
\end{split}
\]
(the factor $2^2$ accounts for passage from oriented to unoriented $\gamma_1,\gamma_2$), with  
\[
I_1(A) = \int_0^A  \left( 2\sinh \frac \ell 2 \right)^2 \,d\ell = e^A-2A- e^{-A} , \quad A>0.
\]
For $A>1$ we bound this by 
\[
I_1(A)<e^A, \quad A>1.
\]
 For $0<A<1$ we use $2\sinh \frac \ell 2\leq 2\ell$ if $0\leq \ell \leq 1$ to bound
 \[
 I_1(A)\leq 4\int_0^A \ell^2 \,d\ell \ll A^3.  
 \]
We obtain that the contribution of SNS pairs is bounded by 
\[
\left( \sum_{q\leq k \leq \log X} X^{1/k}  + \sum_{k>\log X}  \frac{(\log X)^3}{k^3} \right)^2\ll \left(X^{1/q} + \log X \right)^2\ll X^{2/q}
\]
which is $O(X)$ since $q\geq 2$, hence negligible relative to the diagonal term of $X\log X$. 

 \subsection{All other pairs}\label{subsec:all other pairs}

It remains to consider all remaining pairs of primitive geodesics.
These consist of pairs for which at least one geodesic is not simple
non-separating, or both are simple non-separating but intersect, or
are disjoint with disconnected complement. Since all summands are
non-negative, we may discard the congruence condition
$k_1[\gamma_1]\equiv k_2[\gamma_2]\pmod q$, and are left with the
corresponding sum over all such pairs.

Arguing as in   \cite[\S4.2]{RudSI}, we bound this sum by
$(\log X)^2\cdot N_2'(0,\log X)$, 
where $N_2'(0,\log X)$ denotes the number of pairs of primitive geodesics  $(\gamma_1,\gamma_2)$  with length at most $\log X$, 
 such that either one of the geodesics is not simple non-separating, or the two intersect, or they are disjoint with disconnected complement. 
 Each geodesic contributes at most $\Lambda(\gamma)\leq \log X$, so each pair contributes at most $(\log X)^2$. 
 Mirzakhani and Petri \cite[proof of Proposition 4.2 and Proposition 4.5]{MP} 
show that 
\[
\EWp\left(N_2'\left(0,\log X \right)\right)=O_X(1/g),
\]
and therefore the contribution of all remaining pairs is
$O_X(1/g)$, which vanishes in the large genus limit.

\subsection{The mean of $G_M(X,q)$}

We deduce that as $X\to \infty$,
\[
\lim_{g\to \infty} \EWp \left(J_M(X,q) \right)  \sim \begin{cases} X\log X, &q>2,\\ 2 X\log X, & q=2. \end{cases} 
\]
Since 
\[
G_M(X,q) =  J_M(X,q) -  \frac {\Psi_M(X)^2 }{q^{2g}} 
\]
taking the expected value $\lim_{g\to \infty}\EWp$, 
we find that  
\[
\lim_{g\to \infty} \EWp \left(G_M(X,q)  \right) = 
\lim_{g\to \infty} \EWp \left(J_M(X,q) \right)  - \lim_{g\to \infty} \frac  {\EWp \left( \Psi_M(X)^2 \right) }{q^{2g}} .
\]
The term $\lim_{g\to \infty} \EWp \left( \Psi_M(X)^2 \right) $ depends only on $X$ (see Lemma~\ref{lem:mean square of Psi} below) and dividing by $q^{2g}$ kills it in the limit $g\to \infty$, and so we find that as $X\to \infty$  
\[
\lim_{g\to \infty} \EWp \left(G_M(X,q)  \right)=\lim_{g\to \infty} \EWp \left(J_M(X,q) \right) \sim \begin{cases} X\log X, & q>2,\\ 2X\log X, & q=2. \end{cases}
\]

\section{The mean square of $\Psi_M(X)$}

We need to know that $\lim_{g\to \infty}\EWp(|\Psi_M(X)|^2)<\infty$. In fact we can compute it. 
 The result is 
\begin{lem}\label{lem:mean square of Psi}
As $X\to \infty$, 
\[
\lim_{g\to \infty}\EWp(|\Psi_M(X)|^2)    \sim X^2.
\]
\end{lem}
\begin{proof}
We write
\[
\Psi_M(X) =2\sum_\gamma F(\ell_\gamma) ,
\]
where the sum is over unoriented geodesics, and 
\[
F(\ell) = \ell\,\left\lfloor \frac {\log X}{\ell} \right\rfloor = \ell \sum_{n\geq 1} \mathbf 1_{J(n)}(\ell) 
\]
with $J(n) = (0,\frac{\log X}{n}]$.  
Then
\[
\EWp(|\Psi_M(X)|^2)  = 4\EWp\left(\sum_{(\gamma_1,\gamma_2)} F(\ell_1) F(\ell_2) \right).
\]

Arguing as in \S~\ref{subsec:all other pairs},  the only geodesics which contribute something not bounded by $O_X(1/g)$ are  diagonal orbits $\gamma_1=\gamma_2$ with $\gamma_1$ simple non-separating, and SNS pairs $(\gamma_1,\gamma_2)$ with $\gamma_1, \gamma_2$ disjoint simple curves so that the complement of $\gamma_1\cup \gamma_2$ is connected. 

The contribution of the diagonal pairs is evaluated similarly to that in \S~\ref{subsec:diagonal pairs}  and is asymptotic to $X\log X$ for $X\gg 1$. 

For the SNS off-diagonal pairs, we use Mirzakhani's integration formula \eqref{eq:double sum SNS} to obtain 
\[
\lim_{g\to \infty} \EWp(  \off_{SNS})  =   4 \left( \frac 12\int_{0}^{\infty}   F(\ell) \left( \frac{\sinh(\ell/2)}{\ell/2} \right)^2 \,\ell \,d\ell \right)^2.	
\]
We have
\[
\begin{split}
 \int_{0}^{\infty}   F(\ell) \left( \frac{\sinh(\ell/2)}{\ell/2} \right)^2 \,\ell \,d\ell &= 
\int_0^{\log X} \ell \sum_{n\geq 1}  \mathbf 1_{J(n)}(\ell) \left( \frac{\sinh(\ell/2)}{\ell/2} \right)^2 \,\ell \,d\ell 
\\
& = \sum_{n\geq 1} \int_0^{\frac{\log X}{n} } \left(2\sinh \frac \ell 2 \right)^2 d\ell .
\end{split}
\]

We have
\[
 \int_0^{A } \left(2\sinh \frac \ell 2 \right)^2 d\ell =e^A-2A-e^{-A}. 
\]
For $n=1$, we therefore obtain
\[
\int_0^{ \log X  } \left(2\sinh \frac \ell 2 \right)^2 d\ell \sim X.
\]
For $2\leq n\leq \log X$, we bound the integral by $e^A$ to obtain
\[
\sum_{2\leq n\leq \log X} \int_0^{\frac{\log X}{n} } \left(2\sinh \frac \ell 2 \right)^2 d\ell 
<\sum_{2\leq n\leq \log X}  X^{\frac 1n} \ll X^{1/2}.
\]
For $n>\log X$, we use $2\sinh \frac \ell 2\leq \ell$ if $\ell\in (0,1]$ so that the integral is bounded by 
\[
\int_0^{\frac{\log X}{n} } \left(2\sinh \frac \ell 2 \right)^2 d\ell \leq  \int_0^{\frac{\log X}{n} } \ell^2 \,d\ell = \frac{(\log X)^3}{n^3}
\]
and the sum of the integrals is bounded by 
\[
\sum_{n>\log X} \frac{(\log X)^3}{n^3}\ll \log X. 
\]

Therefore we find
\[
 \int_{0}^{\infty}   F(\ell) \left( \frac{\sinh(\ell/2)}{\ell/2} \right)^2 \,\ell \,d\ell \sim X ,
\]
and thus
\[
\lim_{g\to \infty} \EWp \left( |\Psi_M(X) |^2\right)\sim X^2
\]
as $X\to \infty$. 
\end{proof}

\begin{remark} The computation actually shows that 
\[
\lim_{g\to \infty} \EWp\left( \Psi_M(X) \right) \sim X. 
\]
\end{remark}

\section{Comparison with the theory of primes in arithmetic progressions}\label{sec:Hooley et al}
We discuss the distribution of prime numbers in arithmetic progressions.
We set
\begin{equation*}
  \psi(X;Q,A) := \sum_{\substack{n\leq X\\ n=A\bmod Q}} \Lambda(n) 
\end{equation*}
where the von Mangoldt function is $\Lambda(n) = \log p$ if $n=p^k>1$ is a power of a prime $p$, and zero otherwise. 
The Prime Number Theorem for arithmetic progressions states that for a fixed modulus
$Q$,
\begin{equation*}
   \psi(X;Q,A)\sim \frac{X}{\phi(Q)},\quad \mbox{ as }X\to \infty 
   \;.
\end{equation*}

Define
\begin{equation*}
  G(X,Q)=\sum_{\substack{A\bmod Q\\ \gcd(A,Q)=1}} \left| \psi(X;Q,A)-\frac X{\phi(Q)}
  \right|^2 \;.
\end{equation*}
In the ``trivial'' regime $X<Q$, Friedlander and Goldston \cite{FG}  showed that 
\begin{equation}\label{eq:FG trivial}
G(X,Q)\sim X\log X, \quad Q^\varepsilon<X<Q. 
\end{equation}
Hooley \cite{HooleyICM} conjectured that under some (unspecified)
conditions,
\begin{equation}\label{Hooley conj}
  G(X,Q) \sim X\log Q \;.
\end{equation}
Friedlander and Goldston \cite{FG}   argued  that \eqref{Hooley conj}  holds in the range
$X^{1/2+\epsilon}<Q<X$,  assuming a
Hardy-Littlewood conjecture with small remainders.  The asymptotic  \eqref{Hooley conj} is now believed to hold for $X^\varepsilon<Q<X^{1-\varepsilon}$  (see \cite{Fiorilli} for a finer version).  

In comparing Theorem~\ref{thm:G} with the  arithmetic results above, it is important to take account of the order of limits. For fixed $X$ and $q$, we first let $g\to \infty$. The number of homology classes modulo $q$ is $q^{2g}$, and hence tends to infinity, whereas the number of closed geodesics of norm at most $X$ which contribute in the large-genus limit remains bounded in terms of $X$. Thus almost all homology classes contain no such geodesic. In this sense the large-genus limit places us in a sparse, or “trivial”, regime, analogous to the arithmetic range $X<Q$. From this perspective, the asymptotic $X\log X$ in Theorem~\ref{thm:G} is consistent with the arithmetic result \eqref{eq:FG trivial} of Friedlander and Goldston.

A different regime arises in recent work of Pirani \cite{Pirani homology}, who replaces reduction modulo $q$ by reduction modulo a sublattice $\Lambda\subset H_1(M,\Z)$  of index $q$. There are then only $q$  homology classes, rather than $q^{2g}$, and when $q=o(X/\log X)$ the number of primitive geodesics in each class tends to infinity; this therefore provides a geometric analogue of the nontrivial range $Q<X$ in the arithmetic problem.

The distinction between these two regimes can be viewed as a question of how a collection of geodesics is distributed among a varying number of homology classes. In the next section we introduce a simple random allocation model which treats these two regimes in a common framework. 

 \section{A random model}\label{sec:random}

We introduce a simple statistical model, based on the random allocation of weighted balls, which exhibits behavior analogous to that of the geodesic problem.

We are given an infinite set of ``balls", corresponding to the unoriented closed geodesics $\{\gamma_j\}$, 
  with weights (the lengths) $0<\ell_1\leq \ell_2 \leq \ldots $, satisfying
\[
\sum_{\ell_j\leq L}\ell_j \sim \frac 12 X, \qquad 
\sum_{\ell_j\leq L}\ell_j^2 \sim \frac 12 X\log X
\]
where
\[
L:=\log X,
\]
motivated by the corresponding   asymptotics for unoriented primitive geodesics on a hyperbolic surface. 

To each geodesic we attach its two orientations $\gamma^{\pm 1}$ and throw them randomly into a finite abelian group $\mathcal H$ of  $H$ ``boxes", corresponding to the homology classes mod $q$, when $H=q^{2g}$, or to the quotient of size $q$ given by $H_1(M,\Z)/\Lambda$, where $\Lambda$ is a sublattice of index $q$ as in the work of Noam Pirani \cite{Pirani homology}, where $H=q$. 
For simplicity, we assume that $H$ is odd, corresponding to picking $q$ odd in the above scenarios. 
For each $j$, choose $[\gamma_j]$  independently and uniformly from $\mathcal H$, and set
$$[\gamma_j^{-1}]=-[\gamma_j].$$

We define the  counting function  
\[
\Theta_X(\alpha) = \sum_{\ell_j\leq L} \ell_j\Big(\mathbf 1_j(\alpha)+\mathbf 1_j(-\alpha)\Big). 
\]
Note that  
\[
\Theta_X(\alpha) =\Theta_X(-\alpha) ,
\]
and  that the mean value over all $\alpha$ is deterministic:
\[
\frac 1H \sum_{\alpha\in \mathcal H} \Theta_X(\alpha) = \frac 1H\sum_{\ell_j\leq L} 2\ell_j=:\overline{\Theta}.
\]

This models the weighted counting function of oriented closed geodesics in a homology class. It is a weighted version of the classical occupancy problem, or the random allocation model, see e.g. \cite[II.11, problems 5-19 and IV.2]{Feller}.

We now set
\[
G(X):=\sum_{\alpha\in \mathcal H} \left( \Theta_X(\alpha) - \overline{\Theta} \right)^2.
\]
This is our random model of the unnormalized variance. 
\begin{prop}
    The mean value of $G(X)$ is
    \[
    \E(G(X)) = \left( 1-\frac 1H \right) \, 2\sum_{\ell_j\leq L} \ell_j^2
    \]
    and hence as $X\to \infty$,
    \[
    \E(G(X))\sim \left( 1-\frac 1H \right)X\log X.
    \]
\end{prop}
Notice that at the level of the expected variance the model gives the same leading asymptotic $X\log X$  whenever $H=H(X)\to \infty$, independently of whether the boxes are sparsely or densely occupied. The distinction between the two regimes is instead reflected in the typical occupancy of an individual box. 

\begin{proof}

We have 
\[
\E(G) = \sum_\alpha \E\left(\left(\Theta_X(\alpha)-\overline{\Theta}\right)^2\right) = \sum_\alpha \Var(\Theta_X(\alpha)) 
\]
 so that we want to compute the variance of $\Theta_X(\alpha)$.
 
 Let    
\[
\mathbf 1_{j}(\alpha) = \begin{cases} 1, &[\gamma_j]=\alpha,\\ 0,& {\rm else}. \end{cases}
\]
Then  
\[
\Prob(\mathbf 1_{j}(\alpha)=1) = \frac 1H
\]
and so
\[
\E( \mathbf 1_{j}(\alpha)) =  \frac 1H.
\]

We assume that $H$ is odd, so the only option for $\alpha=-\alpha$ is $\alpha=0$. Hence the events $[\gamma_j]=\alpha$ and $[\gamma_j]=-\alpha$ are disjoint for $\alpha\neq 0$. 
Therefore we have for $\alpha\neq 0$, 
\[
\E(\Theta_X(\alpha)) = \frac 2H\sum_{\ell_j\leq L}\ell_j,
\]
while for $\alpha=0$, both orientations contribute to the same class, so that 
\[
\E(\Theta_X(0)) = \frac 2H\sum_{\ell_j\leq L}\ell_j .
\]

For $\alpha\neq 0$, write 
\[
Y_{j,\alpha} = \ell_j\Big(\mathbf 1_j(\alpha)+\mathbf 1_j(-\alpha)\Big).
\]
Then $Y_{j,\alpha}= \ell_j$ with probability $2/H$, and is zero otherwise, hence
\[
\Var(Y_{j,\alpha}) =\ell_j^2 \left( \frac 2H-\frac{4}{H^2}\right).
\] 
Since the different balls (the $j$'s) are independent, we find
\[
\Var \Theta_X(\alpha) = \left( \frac 2H-\frac{4}{H^2}\right)\sum_{\ell_j\leq L} \ell_j^2, \quad \alpha\neq 0.
\]

For $\alpha=0$, we have 
\[
Y_{j,0}=2\ell_j \mathbf 1_j(0)
\]
and so $Y_{j,0}= 2\ell_j$ with probability $1/H$, and is zero otherwise, and 
\[
\Var(Y_{j,0}) =  \left( \frac 1H-\frac{1}{H^2}\right)(2\ell_j)^2,
\]
hence
\[
\Var \Theta_X(0) = 4\left( \frac 1H-\frac{1}{H^2}\right)\sum_{\ell_j\leq L} \ell_j^2.
\]

Therefore we find 
\[
\begin{split} 
\E(G(X)) &=    \sum_\alpha\Var(  \Theta_X(\alpha))
\\
&= 
(H-1)(\frac 1H-\frac 2{H^2}) 2\sum_{\ell_j\leq L} \ell_j^2 + 4\left( \frac 1H-\frac{1}{H^2}\right)\sum_{\ell_j\leq L} \ell_j^2 
\\
&= (1-\frac  1H) \cdot 2\sum_{\ell_j\leq L} \ell_j^2 .
\end{split}
\]
 \end{proof}

 \bigskip

\noindent{\bf Acknowledgements.}   We thank Noam Pirani for helpful discussions. This research was supported by  the ISF-NSFC joint research program (Grant No. 3109/23).  

\bigskip

\noindent{\bf Acknowledgement of AI use.} The author used ChatGPT (OpenAI) as a conversational tool during the preparation of this manuscript. It was used to discuss the interpretation and presentation of the results, in particular  the random allocation model, as well as to assist with exposition and proofreading. All mathematical claims and arguments were independently verified by the author, who takes full responsibility for the contents of the paper.

\end{document}